\documentclass[oneside,british]{amsart}
\usepackage[T1]{fontenc}
\usepackage[latin9]{inputenc}
\usepackage{amstext}
\usepackage{amsthm}
\usepackage{amssymb}

\makeatletter
\numberwithin{equation}{section}
\numberwithin{figure}{section}
\theoremstyle{plain}
\newtheorem{thm}{\protect\theoremname}[section]
\newtheorem{lemma}{\protect\lemmaname}[section]

\theoremstyle{remark}
\newtheorem*{claim*}{\protect\claimname}

\newtheorem{definition}{Definition}

\makeatother

\usepackage{babel}
\providecommand{\claimname}{Claim}
\providecommand{\theoremname}{Theorem}
\providecommand{\lemmaname}{Lemma}
\newcommand{\R}{\mathbb{R}}

\makeatother

\usepackage{babel}
\begin{document}

\title{Dimension of ergodic measures projected onto self-similar sets with
overlaps}

\author{Thomas Jordan}

\address{Thomas Jordan\\
School of Mathematics\\
University of Bristol\\
Bristol}

\email{thomas.jordan@bristol.ac.uk}

\author{Ariel Rapaport}

\address{Ariel Rapaport\\
Department of Pure Mathematics and Mathematical Statistics\\
 Centre for Mathematical Sciences\\
Wilberforce Road\\
Cambridge CB3 0WB}

\email{ar977@dpmms.cam.ac.uk}


\maketitle
\begin{abstract}
For self-similar sets on $\R$ satisfying the exponential separation
condition we show that the dimension of natural projections of shift invariant ergodic measures is equal to $\min\{1,\frac{h}{-\chi}\}$, where $h$ and $\chi$ are the entropy and Lyapunov exponent respectively. The
proof relies on Shmerkin's recent result on the $L^{q}$ dimension
of self-similar measures. We also use the same method to give results
on convolutions and orthogonal projections of ergodic measures projected
onto self-similar sets. 
\end{abstract}

\section{\label{sec:Introduction-and-statement}Introduction and statement
of results}

The dimension of self-similar measures on the line has been the subject
of much attention going back over forty years, since \cite{Hu}. While
the dimension of self-similar measures is well understood when the
open set condition is satisfied, it has been a long standing problem
to see how the dimension behaves when the condition is not satisfied.
Hochman, in \cite{Ho}, made significant progress by showing that
the dimension of self-similar measures can be found as long as an
exponential separation condition is satisfied, which is a much weaker
condition than the open set condition.

Self-similar measures can be thought of as the projection of Bernoulli
measures from a shift space to the self-similar set. So it is also
possible to consider the question of what happens when general ergodic
measures are projected. In the non-overlapping case it is possible
to easily adapt the standard proof to obtain that the dimension is
given by the ratio of the entropy to the Lyapunov exponent, a result
which can also be seen in several other settings for example \cite{Ma}.

In the overlapping case it is easy to see that the ratio of entropy
with Lyapunov exponent is always an upper bound (see section 3 of
\cite{SSU} or Theorem 2.8 in \cite{FH}, where in addition it is
shown that such measures are exact dimensional). In Theorem 7.2 in
\cite{SSU} this is also shown to be a lower bound almost everywhere
for certain families satisfying a transversality condition. However
the techniques used by Hochman in the exponential separation case
for self-similar measures do not apply, since they rely on the convolution
structure of self-similar measures. More precisiely when the self-similar measure is homogeneous, that is when all of the contractions are the same, it is possible to represent it as a convolution of an arbitrarily small copy of itself with some other measure $\nu$ on $\mathbb{R}$. Outside of the homogeneous case, it is possible to obtain such a representation by taking $\nu$ to be a measure on the affine group of $\mathbb{R}$.

Fortunately it turns out that the result of Shmerkin \cite{Sh}, on
the $L^{q}$ dimension of self-similar measures for $q>1$, can be
used to give the dimension of the projection of arbitrary ergodic
measures. The ideas used involve an analysis of numbers of intersections
of cylinders, which are similar to the ideas introduced by Rams in
\cite{Ra}. In addition, similar ideas combined with other results
from \cite{Sh} can be used to give a different proof of a result
of Hochman and Shmerkin on the dimension of convolutions of times
$n$ and times $m$ invariant measures. In particular the result in Theorem 1.3 of \cite{HS} on the convolution of times $n$, times $m$ invariant measures is a special case of Theorem \ref{thm:conv} in this paper. In section 4 we show how the same ideas can be used to give a result on the orthogonal
projections of ergodic measures supported on self-similar sets in
the plane.

\subsection*{Notation}

Before stating our main result we need to state our setting formally
and fix the notation we will be using. In what follows the base of
the $\log$ and $\exp$ functions is always $2$, so that $\exp(a)=2^{a}$
for $a\in\mathbb{R}$. This means our definitions of entropy and Lyapunov
exponent are slightly different to usual, where the usual exponential
and logarithm are used, but fits in more with the use of entropy dimension
used in \cite{Ho} and \cite{Sh}.

Let $\Lambda$ be a finite nonempty set, and for each $\lambda\in\Lambda$
fix $0<|r_{\lambda}|<1$ and $a_{\lambda}\in\mathbb{R}$. Let
\[
\Phi=\{\varphi_{\lambda}(x)=r_{\lambda}x+a_{\lambda}\}_{\lambda\in\Lambda}
\]
be the associated self-similar iterated function system (IFS) on $\mathbb{R}$.
Let $K$ be the attractor of $\Phi$, i.e. $K$ is the unique nonempty
compact subset of $\mathbb{R}$ with 
\[
K=\cup_{\lambda\in\Lambda}\varphi_{\lambda}(K)\:.
\]

Write $\Omega=\Lambda^{\mathbb{N}}$ and let $\sigma:\Omega\rightarrow\Omega$
be the left shift. Given $n\ge1$ and $\lambda_{1}...\lambda_{n}=w\in\Lambda^{n}$
write $[w]\subset\Omega$ for the cylinder set corresponding to $w$,
$r_{w}$ for $r_{\lambda_{1}}\cdot...\cdot r_{\lambda_{n}}$, and
$\varphi_{w}$ for $\varphi_{\lambda_{1}}\circ...\circ\varphi_{\lambda_{n}}$.
For $(\omega_{k})_{k\ge0}=\omega\in\Omega$ set $\omega|_{n}=\omega_{0}...\omega_{n-1}\in\Lambda^{n}$.
Let $\Pi:\Omega\rightarrow K$ be the coding map for $\Phi$, i.e.
\[
\Pi\omega=\underset{n\rightarrow\infty}{\lim}\varphi_{\omega|_{n}}(0)\text{ for }\omega\in\Omega.
\]

We will always assume that our system satisfies an exponential separation
condition introduced by Hochman in \cite{Ho}. We define the distance
between two affine maps $g_{i}(x)=r_{i}x+a_{i}$ on $\mathbb{R}$
as, 
\[
d(g_{1},g_{2})=\left\{ \begin{array}{lll}
|a_{1}-a_{2}| & \text{ if } & r_{1}=r_{2}\\
\infty & \text{ if } & r_{1}\neq r_{2}
\end{array}\right..
\]
It is easy to see that the following definition is equivalent to the
one given in \cite[Section 6.4]{Sh}.

\begin{definition} \label{def:exp_sep} We say that the IFS $\Phi$
has exponential separation if there exist $c>0$ and an increasing
sequence $\{n_{j}\}_{j\geq1}\subset\mathbb{N}$ such that, 
\[
d(\varphi_{w_{1}},\varphi_{w_{2}})\geq c^{n_{j}}\text{ for all }j\geq1\text{ and }w_{1},w_{2}\in\Lambda^{n_{j}}\text{ with }w_{1}\ne w_{2}.
\]
\end{definition} 

This condition is satisfied for instance if $\{r_\lambda\}_{\lambda\in \Lambda}$ and $\{a_\lambda\}_{\lambda\in \Lambda}$ are all algebraic numbers and the maps in $\Phi$ generate a free semigroup. Additionally, in  \cite[Theorem 1.8]{Ho} and \cite[Theorem 1.10]{Ho2} Hochman has shown that in quite general parametrized families of self-similar iterated function systems, the exponential separation condition holds  outside of a set of parameters of packing and Hausdorff co-dimension at least $1$.

For $\delta>0$ and $x\in\mathbb{R}$ write $B(x,\delta)$ for the
interval $[x-\delta,x+\delta]$. A Borel probability measure $\theta$
on $\mathbb{R}$ is said to be exact dimensional if there exists a
number $s\geq0$ with, 
\[
\underset{\delta\downarrow0}{\lim}\:\frac{\log\theta(B(x,\delta))}{\log\delta}=s\text{ for \ensuremath{\theta}-a.e. \ensuremath{x\in\mathbb{R}}},
\]
in which case we write $\dim\theta=s$.

Given a Borel probability measure $\mu$ on $\Omega$ we write $\Pi\mu$
for the push-forward of $\mu$ by $\Pi$. Assuming $\mu$ is $\sigma$-invariant
and ergodic, it follows from \cite[Theorem 2.8]{FH} that $\Pi\mu$
is exact dimensional. We write $h_{\mu}$ for the entropy of $\mu$
and $\chi_{\mu}$ for its Lyapunov exponent with respect to $\{r_{\lambda}\}_{\lambda\in\Lambda}$,
i.e. 
\[
\chi_{\mu}=\sum_{\lambda\in\Lambda}\mu[\lambda]\log|r_{\lambda}|\:.
\]

\subsection*{Main result and structure of the paper}

\begin{thm} \label{thm:main} Suppose that $\Phi$ has exponential
separation, and let $\mu$ be a $\sigma$-invariant and ergodic probability
measure on $\Omega$. Then, 
\[
\dim\Pi\mu=\min\{1,\frac{h_{\mu}}{-\chi_{\mu}}\}\:.
\]
\end{thm} The proof of Theorem \ref{thm:main} is given in the next
section. We first construct suitable self-similar measures, and apply
Shmerkin's results on the $L^{q}$ dimension to these measures. We
then show that these results, together with the connection between
the self-similar and ergodic measures, yield that the dimension can
only drop by an amount which can be made arbitrarily small. For a full definition of $L^q$ dimensions of a measure we refer the reader to section 1.3 in \cite{Sh}. The key result we will be making use of connects $L^q$ dimensions to bounds on the local dimension and is Lemma 1.7 in \cite{Sh}.

In the rest of the paper we state some other applications of this
method to convolutions of ergodic measures and to orthogonal projections
of ergodic measures on the plane. 

\section{Proof of Theorem \ref{thm:main}.}

Fix a $\sigma$-invariant and ergodic measure $\mu$ on $\Omega$,
and write $h$ for $h_{\mu}$ and $\chi$ for $\chi_{\mu}$. We start
with the construction of suitable Bernoulli measures. Let $\beta=\min\{1,\frac{h}{-\chi}\}$,
and in order to obtain a contradiction assume that $\dim\Pi\mu<\beta$.
In particular we have $h>0$. Let $0<\epsilon<\beta-\dim\Pi\mu$ be
small in a manner depending on $\Phi$ and $\mu$, let $\delta>0$
be small with respect to $\epsilon$, and let $m\ge1$ be large with
respect to $\delta$.

Write, 
\[
\mathcal{W}=\{w\in\Lambda^{m}\::\:2^{-m(h+\delta)}\le\mu[w]\le2^{-m(h-\delta)}\text{ and }|r_{w}|\ge2^{m(\chi-\delta)}\}\:.
\]
By combining Egorov's Theorem with the Shannon-Macmillan-Breiman Theorem
and the ergodic theorem (applied to the function $\omega\mapsto\log r_{\omega_{0}}$),
it can be seen that by taking $m$ sufficiently large we can obtain
that 
\begin{equation}
\mu(\cup_{w\in\mathcal{W}}[w])>1-\delta\:.\label{eq:mass of good cylinders}
\end{equation}
For $w\in\Lambda^{m}$ set 
\[
p_{w}=\begin{cases}
\mu[w]\cdot c & \text{if }w\in\mathcal{W}\\
2^{-m\epsilon^{-1}}\cdot c & \text{otherwise}
\end{cases},
\]
where $c>0$ is chosen so that $\sum_{w\in\Lambda^{m}}p_{w}=1$. By
(\ref{eq:mass of good cylinders}) and by assuming that $\epsilon^{-1}>\log|\Lambda|$ it follows
that $1/2\le c\le2$. Write $p=(p_{w})_{w\in\Lambda^{m}}$ and let
$\nu$ be the measure on $\Omega$ with, 
\[
\nu[w_{1}...w_{l}]=p_{w_{1}}\cdot...\cdot p_{w_{l}}\text{ for each }w_{1},...,w_{l}\in\Lambda^{m}\:.
\]
We now relate the expected behaviour of the $L^{q}$ dimension of
$\Pi\nu$ to the expected dimension of $\Pi\mu$. Write $q$ for $\delta^{-1}$
and let $\tau>0$ be the unique solution to, 
\[
\sum_{w\in\Lambda^{m}}p_{w}^{q}|r_{w}|^{-\tau}=1\:.
\]
\begin{lemma} By taking $\epsilon$ and $\delta$ to be small enough, and $m$ to be large enough, we may assume that, 
\begin{equation}
\frac{\tau}{q-1}\ge\frac{h}{-\chi}-O(\delta)\:.\label{eq:lb of tau}
\end{equation}
\end{lemma} \begin{proof} Write $\rho_{1}=\underset{\lambda\in\Lambda}{\min}|r_{\lambda}|$,
$\rho_{2}=\underset{\lambda\in\Lambda}{\max}|r_{\lambda}|$, $\Vert p\Vert_{q}^{q}=\sum_{w\in\Lambda^{m}}p_{w}^{q}$,
and $\Vert p\Vert_{\infty}=\underset{w\in\Lambda^{m}}{\max}\:p_{w}$.
Then, 
\[
0\ge\log\left(\sum_{w\in\Lambda^{m}}p_{w}^{q}\rho_{2}^{-m\tau}\right)=\log\Vert p\Vert_{q}^{q}-m\tau\log\rho_{2}\:.
\]
We may assume $\delta<h$, hence 
\[
\Vert p\Vert_{q}^{q}\ge\Vert p\Vert_{\infty}^{q}\ge2^{-mq(h+\delta)}\ge2^{-2mqh},
\]
and so, 
\[
\tau\le\frac{\log\Vert p\Vert_{q}^{q}}{m\log\rho_{2}}\le\frac{-2qh}{\log\rho_{2}}\:.
\]
From this and by the definitions of $\mathcal{W}$ and $p$, 
\begin{eqnarray*}
1 & = & \sum_{w\in\Lambda^{m}}p_{w}^{q}\cdot|r_{w}|^{-\tau}\\
 & \le & \sum_{w\in\mathcal{W}}p_{w}^{q}\cdot2^{m\tau(\delta-\chi)}+\sum_{w\in\Lambda^{m}\setminus\mathcal{W}}c^{q}2^{-m\epsilon^{-1}q}\cdot\rho_{1}^{-m\tau}\\
 & \le & 2^{m\tau(\delta-\chi)}\Vert p\Vert_{q}^{q}+\sum_{w\in\Lambda^{m}\setminus\mathcal{W}}\exp\left(q(1-m\epsilon^{-1}+2mh\frac{\log\rho_{1}}{\log\rho_{2}})\right)\\
 & \le & 2^{m\tau(\delta-\chi)}\Vert p\Vert_{q}^{q}+\exp\left(m\log|\Lambda|+q(1-m\epsilon^{-1}+2mh\frac{\log\rho_{1}}{\log\rho_{2}})\right)\:.
\end{eqnarray*}
By choosing $\epsilon$ small enough in a manner depending on $\Phi$
and $\mu$ we may clearly assume that, 
\[
m\log|\Lambda|+q(1-m\epsilon^{-1}+2mh\frac{\log\rho_{1}}{\log\rho_{2}})<-1\:.
\]
Hence 
\[
1/2\le2^{m\tau(\delta-\chi)}\Vert p\Vert_{q}^{q},
\]
and so 
\[
\tau\ge\frac{-1-\log\Vert p\Vert_{q}^{q}}{m(\delta-\chi)}\:.
\]
We also have, 
\begin{eqnarray*}
\Vert p\Vert_{q}^{q} & \le & \Vert p\Vert_{\infty}^{q-1}\sum_{w\in\Lambda^{m}}p_{w}\\
 & \le & c^{q-1}\exp(-m(h-\delta)(q-1))\\
 & \le & \exp((q-1)(1-m(h-\delta)))\:.
\end{eqnarray*}
Hence by assuming that $m$ is large enough with respect to $\delta$,
\[
\frac{\tau}{q-1}\ge\frac{h-\delta}{\delta-\chi}-\delta,
\]
which completes the proof of the lemma. \end{proof}

To apply Shmerkin's result we will need the following lemma. Its proof
is a simple consequence of the fact that $\Phi$ has exponential separation,
and is therefore omitted.

\begin{lemma}\label{nthlevel} The IFS $\{\varphi_{w}\}_{w\in\Lambda^{m}}$
has exponential separation. \end{lemma} We can now use Shmerkin's
result on the $L^{q}$ dimension of self-similar measures with exponential
separation. Fix some $0<\alpha<\min\{\frac{\tau}{q-1},1\}$. \begin{lemma}
There exists $\eta_{0}>0$, which depends on all previous parameters,
such that 
\begin{equation}
\Pi\sigma^{j}\nu(B(x,\eta))\le\eta^{(1-\delta)\alpha}\text{ for all }0\le j<m,\;0<\eta\le\eta_{0}\text{ and }x\in\mathbb{R}\:.\label{eq:all j mass of balls small}
\end{equation}
\end{lemma} \begin{proof} By Lemma \ref{nthlevel} the IFS $\{\varphi_{w}\}_{w\in\Lambda^{m}}$
has exponential separation. Thus from \cite[Theorem 6.6]{Sh} it follows
that the $L^{q}$ dimension of $\Pi\nu$ is equal to $\min\{\frac{\tau}{q-1},1\}$.
Write 
\[
\alpha'=\frac{1}{2}(\alpha+\min\{\frac{\tau}{q-1},1\}),
\]
then by \cite[Lemma 1.7]{Sh} and $q=\delta^{-1}$ it follows that
there exists $\eta_{1}>0$ with, 
\[
\Pi\nu(B(x,\eta))\le\eta^{(1-\delta)\alpha'}\text{ for all }0<\eta\le\eta_{1}\text{ and }x\in\mathbb{R}\:.
\]
\par Let $\eta_{0}>0$ be small with respect to $\eta_{1}$, $m$,
$|\Lambda|$ and $\alpha'-\alpha$. Given a Borel set $E\subset\Omega$
write $\nu|_{E}$ for the restriction of $\nu$ to $E$. For every
$0\le j<m$, $0<\eta\le\eta_{0}$, $x\in\mathbb{R}$, and $u\in\Lambda^{j}$,
\begin{eqnarray*}
\Pi\sigma^{j}(\nu|_{[u]})(B(x,\eta)) & = & \nu\{\omega\in[u]\::\:\Pi\sigma^{j}\omega\in B(x,\eta)\}\\
 & = & \nu\{\omega\in[u]\::\:\varphi_{u}^{-1}\Pi\omega\in B(x,\eta)\}\\
 & = & \nu\{\omega\in[u]\::\:\Pi\omega\in B(\varphi_{u}x,\eta r_{u})\}\\
 & \le & \Pi\nu(B(\varphi_{u}x,\eta r_{u}))\le\eta^{(1-\delta)\alpha'}\:.
\end{eqnarray*}
Hence, 
\[
\Pi\sigma^{j}\nu(B(x,\eta))=\sum_{u\in\Lambda^{j}}\Pi\sigma^{j}(\nu|_{[u]})(B(x,\eta))\le|\Lambda|^{m}\eta^{(1-\delta)\alpha'}<\eta^{(1-\delta)\alpha},
\]
which completes the proof of the lemma. \end{proof} We now need to
relate the behaviour of the Bernoulli measure $\nu$ and our original
ergodic measure $\mu$. Define $f:\Omega\mapsto\R$ by 
\[
f(\omega)=-\frac{1}{m}1_{\mathcal{W}}(\omega|_{m})\log\mu[\omega|_{m}],
\]
for all $\omega\in\Omega$. By the definition of $f$ and $\mathcal{W}$
we have that $\int f\:d\mu\le h+\delta$ . Let $N\ge1$ be large with
respect to all previous parameters. Let $\Omega_{0}$ be the set of
all $\omega\in\Omega$ such that for every $n\ge N$, 
\begin{enumerate}
\item $\mu[\omega|_{nm}]<2^{-nm(h-\delta)}$; 
\item $|r_{\omega|_{nm}}|<2^{nm(\chi+\delta)}$; 
\item \label{Omega0} $\frac{1}{nm}\sum_{k=0}^{nm-1}f(\sigma^{k}\omega)+\frac{1}{\epsilon nm}\sum_{k=0}^{nm-1}1_{\{(\sigma^{k}\omega)|_{m}\notin\mathcal{W}\}}\le h+2\delta(1+\epsilon^{-1})\:.$ 
\end{enumerate}
By $\int f\:d\mu\le h+\delta$ and $\eqref{eq:mass of good cylinders}$,
and since $\mu$ is ergodic, we may assume that $\mu(\Omega_{0})>1/2$.
Note that the fact that $\mu$ is ergodic for $\sigma$ does not necessarily
imply that $\mu$ is ergodic for $\sigma^{m}$, the following lemma
allows us to take care of this. \begin{lemma}\label{nergodic} There
exists a global constant $c_{1}>1$ such that for every $\omega\in\Omega_{0}$
and $n\ge N$, 
\begin{equation}
-\frac{1}{nm}\log\sigma^{j}\nu[\omega|_{nm}]\le h+c_{1}\delta/\epsilon\text{ for some }0\le j<m\:.\label{eq:some j}
\end{equation}
\end{lemma} \begin{proof} Let $\omega\in\Omega_{0}$ and $n\ge N$,
then by partitioning \eqref{Omega0} into $m$ sums we can see there
must exist $0\le j<m$ such that 
\begin{equation}
\frac{1}{n}\sum_{k=1}^{n-1}f(\sigma^{km-j}\omega)+\frac{1}{\epsilon n}\sum_{k=1}^{n-1}1_{\{(\sigma^{km-j}\omega)|_{m}\notin\mathcal{W}\}}\le h+2\delta(1+\epsilon^{-1})\:.\label{eq:some j ergodic avg}
\end{equation}
By the definition of $\nu$, 
\begin{equation}
\sigma^{j}\nu[\omega|_{nm}]=\nu(\sigma^{-j}[\omega|_{m-j}])\cdot\left(\prod_{k=1}^{n-1}\nu[(\sigma^{km-j}\omega)|_{m}]\right)\cdot\nu[(\sigma^{nm-j}\omega)|_{j}]\:.\label{eq:as product}
\end{equation}
Since $p_{w}\ge c2^{-m\epsilon^{-1}}$ for every $w\in\Lambda^{m}$
we may assume that $N$ is sufficiently large so that, 
\[
-\frac{1}{nm}\log\nu(\sigma^{-j}[\omega|_{m-j}])-\frac{1}{nm}\log\nu[(\sigma^{nm-j}\omega)|_{j}]\le\delta/2\:.
\]
From this, (\ref{eq:some j ergodic avg}), (\ref{eq:as product})
and $c\ge1/2$, we now get\par 
\begin{eqnarray*}
-\frac{1}{nm}\log\sigma^{j}\nu[\omega|_{nm}] & \le & -\frac{1}{nm}\sum_{k=1}^{n-1}\log\nu[(\sigma^{km-j}\omega)|_{m}]+\delta/2\\
 & \le & -\frac{1}{nm}\sum_{k=1}^{n-1}\log(p_{(\sigma^{km-j}\omega)|_{m}}/c)+\delta\\
 & = & \frac{1}{n}\sum_{k=1}^{n-1}f(\sigma^{km-j}\omega)+\frac{1}{\epsilon n}\sum_{k=1}^{n-1}1_{\{(\sigma^{km-j}\omega)|_{m}\notin\mathcal{W}\}}+\delta\\
 & \le & h+3\delta(1+\epsilon^{-1}),
\end{eqnarray*}
which completes the proof of the lemma. \end{proof}
We are now ready to complete the proof of the Theorem. For a Borel set $E\subset\Omega$ write 
$\mu_{0}(E)=\frac{\mu(E\cap\Omega_{0})}{\mu(\Omega_{0})}$.
Since $\Pi\mu_{0}\ll\Pi\mu$, it follows by \cite[Theorem 2.12]{Mat} that for $\Pi\mu_{0}$-a.e. $x\in\mathbb{R}$ the limit
\[
\underset{\eta\downarrow0}{\lim}\:\frac{\Pi\mu_{0}(B(x,\eta))}{\Pi\mu(B(x,\eta))}\
\]
exists, and it is positive and finite. Thus, since $\Pi\mu$ is exact dimensional, the same goes for $\Pi\mu_{0}$ with,
\[
\dim\Pi\mu_{0}=\dim\Pi\mu<\beta-\epsilon\:.
\]

Let $n\ge N$ and $x\in\mathbb{R}$ be with, 
\[
\frac{\log\Pi\mu_{0}(B(x,2^{nm\chi}))}{nm\chi}<\beta-\epsilon\:.
\]
Write, 
\[
\mathcal{U}=\{w\in\Lambda^{nm}\::\:[w]\cap\Pi^{-1}(B(x,2^{nm\chi}))\ne\emptyset\text{ and }\mu_{0}[w]>0\}\:.
\]
Since $\mu(\Omega_{0})>1/2$, 
\begin{equation}
2^{nm\chi(\beta-\epsilon)}<\Pi\mu_{0}(B(x,2^{nm\chi}))\le\sum_{w\in\mathcal{U}}\mu_{0}[w]\le2\sum_{w\in\mathcal{U}}\mu[w]\:.\label{eq:lb on sum of mass of cyl}
\end{equation}
For each $w\in\mathcal{U}$ we have $\mu_{0}[w]>0$, hence $\Omega_{0}\cap[w]\ne\emptyset$,
and so $\mu[w]<2^{-nm(h-\delta)}$. From this and (\ref{eq:lb on sum of mass of cyl})
we get, 
\[
2^{nm\chi(\beta-\epsilon)}<2^{1-nm(h-\delta)}\cdot|\mathcal{U}|\:.
\]

For $0\le j<m$ write, 
\[
\mathcal{U}_{j}=\{w\in\mathcal{U}\::\:\sigma^{j}\nu[w]\ge\exp(-nm(h+c_{1}\delta/\epsilon))\}\:.
\]
From (\ref{eq:some j}) and $n\ge N$, and since $\Omega_{0}\cap[w]\ne\emptyset$
for each $w\in\mathcal{U}$, it follows that $\mathcal{U}=\cup_{j=0}^{m-1}\mathcal{U}_{j}$.
Hence there exists $0\le j<m$ with 
\begin{equation}
|\mathcal{U}_{j}|\ge|\mathcal{U}|/m>2^{nm\chi(\beta-\epsilon)}\cdot2^{nm(h-\delta)}\cdot\frac{1}{2m}\:.\label{eq:lb on card U_j}
\end{equation}

Without loss of generality we may assume that $\mathrm{diam}(K)\le1$.
Given $w\in\mathcal{U}_{j}$ we have $\Pi[w]\cap B(x,2^{nm\chi})\ne\emptyset$.
Since $\Omega_{0}\cap[w]\ne\emptyset$, 
\[
\mathrm{diam}(\Pi[w])=\mathrm{diam}(\varphi_{w}(K))\le|r_{w}|<2^{nm(\chi+\delta)},
\]
which implies $[w]\subset\Pi^{-1}(B(x,2^{nm(\chi+2\delta)}))$. Hence,
by the definition of $\mathcal{U}_{j}$, 
\[
\Pi\sigma^{j}\nu(B(x,2^{nm(\chi+2\delta)}))\ge\sigma^{j}\nu(\cup_{w\in\mathcal{U}_{j}}[w])\ge|\mathcal{U}_{j}|\cdot\exp(-nm(h+c_{1}\delta/\epsilon))\:.
\]
From this and (\ref{eq:lb on card U_j}), 
\[
\Pi\sigma^{j}\nu(B(x,2^{nm(\chi+2\delta)}))\ge\frac{1}{2m}\exp\left(nm(\chi(\beta-\epsilon)-O(\delta/\epsilon))\right)\:.
\]

On the other hand, by (\ref{eq:all j mass of balls small}) and by
assuming that $n$ is large enough, 
\[
\Pi\sigma^{j}\nu(B(x,2^{nm(\chi+2\delta)}))\le\exp(nm(\chi+2\delta)(1-\delta)\alpha)\:.
\]
Hence 
\[
\frac{1}{2m}\exp\left(nm(\chi(\beta-\epsilon)-O(\delta/\epsilon))\right)\le\exp(nm(\chi+2\delta)(1-\delta)\alpha),
\]
and so by taking logarithm on both sides, dividing by $nm\chi$, and
letting $n$ tend to $\infty$, we get 
\[
\beta-\epsilon+O(\delta/\epsilon)\ge(1+2\delta/\chi)(1-\delta)\alpha\:.
\]
Now by (\ref{eq:lb of tau}) and since this holds for every $0\le\alpha<\min\{\frac{\tau}{q-1},1\}$,
\begin{equation}
\beta-\epsilon+O(\delta/\epsilon)\ge(1+2\delta/\chi)(1-\delta)\min\{\frac{h}{-\chi}-O(\delta),1\}\:.\label{eq:contradiction}
\end{equation}
Recall that $\delta$ is arbitrarily small with respect to $\epsilon$
and that $\beta=\min\{1,\frac{h}{-\chi}\}$. Hence (\ref{eq:contradiction})
gives a contradiction, and so we must have $\dim\Pi\mu\ge\beta$.
Since it always holds that $\dim\Pi\mu\le\beta$ (see section 3 of
\cite{SSU} or Theorem 2.8 in \cite{FH} for details of how to prove
this), this completes the proof of Theorem \ref{thm:main}.

\section{Convolutions of ergodic measures}

In this section we show how to use the ideas from the proof of Theorem
\ref{thm:main} to prove a result on the convolution of ergodic measures.

For $i=1,2$ let $\Phi_{i}=\{\varphi_{\lambda,i}(x)=r_{i}x+a_{\lambda,i}\}_{\lambda\in\Lambda_{i}}$
be a homogeneous self-similar IFS on $\mathbb{R}$, write $\Omega_{i}=\Lambda_{i}^{\mathbb{N}}$,
let $\Pi_{i}:\Omega_{i}\rightarrow\mathbb{R}$ be the coding map for
$\Phi_{i}$, let $\sigma_{i}:\Omega_{i}\rightarrow\Omega_{i}$ be
the left shift, let $\mu_{i}$ be a $\sigma_{i}$-invariant and ergodic
probability measure on $\Omega_{i}$, and write $h_{i}$ for the entropy
of $\mu_{i}$. We also write $\theta$ for the convolution $\Pi_{1}\mu_{1}*\Pi_{2}\mu_{2}$.

Recall that in Section \ref{sec:Introduction-and-statement} a distance
$d$ was defined between affine maps from $\mathbb{R}$ to $\mathbb{R}$.
We say that $\Phi_{1},\Phi_{2}$ are jointly exponentially separated
if there exist $c>0$ and an increasing sequence $\{n_{j}\}_{j\ge1}\subset\mathbb{N}$
such that,
\[
d(\varphi_{w_{1},i},\varphi_{w_{2},i})\ge c^{n_{j}}\text{ for }i=1,2,\:j\ge1\text{ and }w_{1},w_{2}\in\Lambda_{i}^{n_{j}}\text{ with }w_{1}\ne w_{2}\:.
\]
\begin{thm}\label{thm:conv} Suppose that $\log r_{1}/\log r_{2}\notin\mathbb{Q}$
and that $\Phi_{1},\Phi_{2}$ are jointly exponentially separated.
Then $\theta$ is exact dimensional and, 
\[
\dim\theta=\min\{1,\frac{h_{1}}{-\log r_{1}}+\frac{h_{2}}{-\log r_{2}}\}\:.
\]
\end{thm} In the case of self-similar measures the theorem follows
almost directly from \cite[Theorem 7.2]{Sh}, which is the main ingredient
of our proof. In \cite[Theorem 1.3]{HS} the above result is shown
for systems $\Phi_{i}$ of the form
\[
\{\varphi_{\lambda,i}(x)=x/n_{i}+\lambda t_{i}/n_{i}\}_{\lambda=0}^{n_{i}-1},
\]
where $t_{1},t_{2}>0$ are real and $n_{1},n_{2}$ are positive integers
with $\log n_{1}/\log n_{2}\notin\mathbb{Q}$. Such systems are clearly
jointly exponentially separated (in fact they satisfy the more restrictive
open set condition).

\subsection*{Preparations for the proof of Theorem \ref{thm:conv}}

Given a Borel probability measure $\zeta$ on $\mathbb{R}$ write $\dim_{H}\zeta$ and $\dim_{P}^{*}\zeta$ for its lower Hausdorff and upper packing dimensions. That is,
\[
\dim_{H}\zeta=\sup\{s\ge0\::\:\underset{\eta\downarrow0}{\liminf}\:\frac{\log\zeta(B(x,\eta))}{\log\eta}\ge s\text{ for }\zeta\text{-a.e. }x\in\mathbb{R}\}
\]
and
\[
\dim_{P}^{*}\zeta=\inf\{s\ge0\::\:\underset{\eta\downarrow0}{\limsup}\:\frac{\log\zeta(B(x,\eta))}{\log\eta}\le s\text{ for }\zeta\text{-a.e. }x\in\mathbb{R}\}\:.
\]
Clearly $\dim_{H}\zeta\le\dim_{P}^{*}\zeta$, and $\zeta$ has exact dimension $s$ if and only if $s=\dim_{H}\zeta=\dim_{P}^{*}\zeta$.
Given a Borel set $E\subset\mathbb{R}$ denote its Hausdorff dimension by $\dim_{H}E$. It is well known that,
\begin{equation}
\dim_{H}\zeta=\inf\{\dim_{H}E\::\:E\subset\mathbb{R}\text{ is Borel and }\zeta(E)>0\}\:.\label{eq:alt def of dim_H}
\end{equation}
For further details on these notions see \cite[Section 10]{Fa}.

Recall that the total variation distance between Borel probability
measures $\zeta_{1},\zeta_{2}$ on $\mathbb{R}$ is defined by,
\[
d_{TV}(\zeta_{1},\zeta_{2})=\sup\{|\zeta_{1}(E)-\zeta_{2}(E)|\::\:E\subset\mathbb{R}\text{ is Borel}\}\:.
\]

\begin{lemma}\label{lemma:upper semi cont} 
The function which takes a probability measure $\zeta$ on $\mathbb{R}$
to $\dim_{H}\zeta$ is upper semicontinuous with respect to the total
variation distance.
\end{lemma}

\begin{proof}
Let $\zeta$ be a probability measure on $\mathbb{R}$ and let $s>\dim_{H}\zeta$.
By (\ref{eq:alt def of dim_H}) there exists a Borel set $E\subset\mathbb{R}$
with $\zeta(E)>0$ and $\dim_{H}E<s$. Now suppose that $\xi$ is
another probability measure on $\mathbb{R}$ with $d_{TV}(\zeta,\xi)<\zeta(E)$.
Then,
\[
\xi(E)>\zeta(E)-d_{TV}(\zeta,\xi)>0,
\]
and so by (\ref{eq:alt def of dim_H}),
\[
\dim_{H}\xi\le\dim_{H}E<s\:.
\]
This completes the proof of the lemma.
\end{proof}

\subsection*{Proof of Theorem \ref{thm:conv}}

We let,
\[
\beta=\min\{1,\frac{h_{1}}{-\log r_{1}}+\frac{h_{2}}{-\log r_{2}}\}.
\]
By Theorem \ref{thm:main} it follows that $\Pi_{i}\mu_{i}$ has exact dimension $\min\{1,\frac{h_{i}}{-\log r_{i}}\}$ for $i=1,2$. Thus, it is easy to see that $\Pi_{1}\mu_{1}\times\Pi_{2}\mu_{2}$ has exact dimension,
\[
\min\{1,\frac{h_{1}}{-\log r_{1}}\}+\min\{1,\frac{h_{2}}{-\log r_{2}}\}\:.
\]
Now since $\theta$ is a linear projection of $\Pi_{1}\mu_{1}\times\Pi_{2}\mu_{2}$,
\[
\dim_{P}^{*}\theta\le\min\{1,\dim(\Pi_{1}\mu_{1}\times\Pi_{2}\mu_{2})\}=\beta\:.
\]
Thus it suffices to prove that $\dim_{H}\theta\ge\beta$. Assume by contradiction
that $\dim_{H}\theta<\beta$. Let $0<\epsilon<\beta-\dim_{H}\theta$
be small in a manner depending on $\Phi_{i}$ and $\mu_{i}$, let
$\delta>0$ be small with respect to $\epsilon$, and let $m\ge1$
be large with respect to $\delta$.

For $i=1,2$ write, 
\[
\mathcal{W}_{i}=\{w\in\Lambda_{i}^{m}\::\:2^{-m(h_{i}+\delta)}\le\mu_{i}[w]\le2^{-m(h_{i}-\delta)}\}\:.
\]
We may assume that, 
\begin{equation}
\mu_{i}(\cup_{w\in\mathcal{W}_{i}}[w])>1-\delta\:.\label{eq:mass of good cylinders3}
\end{equation}
For $w\in\Lambda_{i}^{m}$ set 
\[
p_{w,i}=\begin{cases}
\mu_{i}[w]\cdot c_{i} & \text{if }w\in\mathcal{W}_{i}\\
2^{-m\epsilon^{-1}}\cdot c_{i} & \text{otherwise}
\end{cases},
\]
where $c_{i}>0$ is chosen so that $\sum_{w\in\Lambda_{i}^{m}}p_{w,i}=1$.
By (\ref{eq:mass of good cylinders3}) it follows that $1/2\le c_{i}\le2$.
Write $p_{i}=(p_{w,i})_{w\in\Lambda_{i}^{m}}$ and let $\nu_{i}$
be the measure on $\Omega_{i}$ with, 
\[
\nu_{i}[w_{1}...w_{l}]=p_{w_{1},i}\cdot...\cdot p_{w_{l},i}\text{ for each }w_{1},...,w_{l}\in\Lambda_{i}^{m}\:.
\]
For $t>0$ and $x\in\mathbb{R}$ set $S_{t}x=tx$ and $\xi_{t}=\Pi_{1}\nu_{1}*S_{t}\Pi_{2}\nu_{2}$.
Write $q$ for $\delta^{-1}$. Given a Borel probability measure $\zeta$
on $\mathbb{R}$ denote by $D(\zeta,q)$ the $L^{q}$ dimension of
$\zeta$. \begin{lemma} There exists a constant $c_{1}\ge1$, which
depends only on $r_{1},r_{2}$, such that 
\begin{equation}
D(\xi_{t},q)>\beta-c_{1}\delta\text{ for all }t>0\:.\label{eq:L^q dim of xi_t}
\end{equation}
\end{lemma} \begin{proof} For $i=1,2$ we have 
\[
\Vert p_{i}\Vert_{q}^{q}\le\Vert p_{i}\Vert_{\infty}^{q-1}\sum_{w\in\Lambda_{i}^{m}}p_{w,i}\le\exp(-m(h_{i}-\delta)(q-1))\:.
\]
From this and \cite[Theorem 6.2]{Sh}, 
\[
D(\Pi_{i}\nu_{i},q)=\min\{1,\frac{\log\Vert p_{i}\Vert_{q}^{q}}{(q-1)\log r_{i}^{m}}\}\ge\min\{1,\frac{h_{i}-\delta}{-\log r_{i}}\}\:.
\]
From the fact that $\Phi_{i}$ are jointly exponentially separated
it follows easily that the systems $\{\varphi_{w,i}\}_{w\in\Lambda_{i}^{m}}$
are also jointly exponentially separated. From this and the assumption
$\log r_{1}/\log r_{2}\notin\mathbb{Q}$, by \cite[Theorem 7.2]{Sh},
and since $D(S_{t}\Pi_{2}\nu_{2},q)=D(\Pi_{2}\nu_{2},q)$ for $t>0$,
we get
\begin{eqnarray*}
D(\xi_{t},q) & = & \min\{1,D(\Pi_{1}\nu_{1},q)+D(S_{t}\Pi_{2}\nu_{2},q)\}\\
 & \ge & \min\{1,\frac{h_{1}}{-\log r_{1}}+\frac{h_{2}}{-\log r_{2}}\}-O_{r_{1},r_{2}}(\delta),
\end{eqnarray*}
which completes the proof of the lemma. \end{proof} Fix some $0<\alpha<\beta-c_{1}\delta$.
\begin{lemma} There exists $\eta_{0}>0$, which depends on all previous
parameters, such that for every $0\le j_{1},j_{2}<m$, 
\begin{equation}
\Pi_{1}\sigma_{1}^{j_{1}}\nu_{1}*\Pi_{2}\sigma_{2}^{j_{2}}\nu_{2}(B(x,\eta))\le\eta^{(1-\delta)\alpha}\text{ for all }0<\eta\le\eta_{0}\text{ and }x\in\mathbb{R}\:.\label{eq:all j mass of balls small3}
\end{equation}
\end{lemma} \begin{proof} Write, 
\[
T=\{r_{1}^{j_{1}}r_{2}^{-j_{2}}\::\:0\le j_{1},j_{2}<m\}\:.
\]
By \cite[Lemma 1.7]{Sh}, (\ref{eq:L^q dim of xi_t}), and $q=\delta^{-1}$,
it follows that there exists $\eta_{1}>0$ with, 
\[
\xi_{t}(B(x,\eta))\le\eta^{(1-\delta)(\beta-c_{1}\delta)}\text{ for all }t\in T,\:0<\eta\le\eta_{1}\text{ and }x\in\mathbb{R}\:.
\]
\par Let $\eta_{0}>0$ be small with respect to $\eta_{1}$ and all
previous parameters. Let $0\le j_{1},j_{2}<m$, $0<\eta\le\eta_{0}$,
$x\in\mathbb{R}$, $u_{1}\in\Lambda_{1}^{j_{1}}$, and $u_{2}\in\Lambda_{2}^{j_{2}}$.
Write $b=\varphi_{u_{1},1}\circ\varphi_{u_{2},2}^{-1}(0)$, then 
\begin{multline*}
\Pi_{1}\sigma_{1}^{j_{1}}(\nu_{1}|_{[u_{1}]})*\Pi_{2}\sigma_{2}^{j_{2}}(\nu_{2}|_{[u_{2}]})(B(x,\eta))\\
=\nu_{1}\times\nu_{2}\{(\omega_{1},\omega_{2})\in[u_{1}]\times[u_{2}]\::\:\Pi_{1}\sigma_{1}^{j_{1}}\omega_{1}+\Pi_{2}\sigma_{2}^{j_{2}}\omega_{2}\in B(x,\eta)\}\\
=\nu_{1}\times\nu_{2}\{(\omega_{1},\omega_{2})\in[u_{1}]\times[u_{2}]\::\:\varphi_{u_{1},1}^{-1}\Pi_{1}\omega_{1}+\varphi_{u_{2},2}^{-1}\Pi_{2}\omega_{2}\in B(x,\eta)\}\\
\le\nu_{1}\times\nu_{2}\{(\omega_{1},\omega_{2})\::\:\Pi_{1}\omega_{1}+S_{r_{1}^{j_{1}}r_{2}^{-j_{2}}}\Pi_{2}\omega_{2}\in B(\varphi_{u_{1},1}x-b,r_{1}^{j_{1}}\eta)\}\\
=\xi_{r_{1}^{j_{1}}r_{2}^{-j_{2}}}(B(\varphi_{u_{1},1}x-b,r_{1}^{j_{1}}\eta))\le\eta^{(1-\delta)(\beta-c_{1}\delta)}\:.
\end{multline*}
Hence,
\begin{eqnarray*}
\Pi_{1}\sigma_{1}^{j_{1}}\nu_{1}*\Pi_{2}\sigma_{2}^{j_{2}}\nu_{2}(B(x,\eta)) & = & \sum_{u_{1}\in\Lambda_{1}^{j_{1}}}\sum_{u_{2}\in\Lambda_{2}^{j_{2}}}\Pi_{1}\sigma_{1}^{j_{1}}(\nu_{1}|_{[u_{1}]})*\Pi_{2}\sigma_{2}^{j_{2}}(\nu_{2}|_{[u_{2}]})(B(x,\eta))\\
 & \le & |\Lambda_{1}|^{m}|\Lambda_{2}|^{m}\eta^{(1-\delta)(\beta-c_{1}\delta)}\le\eta^{(1-\delta)\alpha},
\end{eqnarray*}
which completes the proof of the lemma. \end{proof} For $i=1,2$
and $\omega\in\Omega_{i}$ set 
\[
f_{i}(\omega)=-\frac{1}{m}1_{\mathcal{W}_{i}}(\omega|_{m})\log\mu_{i}[\omega|_{m}],
\]
then $\int f_{i}\:d\mu_{i}\le h_{i}+\delta$. Let $N\ge1$ be large
with respect to all previous parameters. For $n\ge1$ write $n_{i}=\left\lceil \frac{n}{-\log r_{i}}\right\rceil $.
Let $\Omega_{0,i}$ be the set of all $\omega\in\Omega_{i}$ such
that for every $n\ge N$, 
\begin{itemize}
\item $\mu_{i}[\omega|_{n_{i}m}]<2^{-n_{i}m(h_{i}-\delta)}$;
\item $\frac{1}{n_{i}m}\sum_{k=0}^{n_{i}m-1}f_{i}(\sigma_{i}^{k}\omega)+\frac{1}{\epsilon n_{i}m}\sum_{k=0}^{n_{i}m-1}1_{\{(\sigma_{i}^{k}\omega)|_{m}\notin\mathcal{W}_{i}\}}\le h_{i}+2\delta(1+\epsilon^{-1})\:.$
\end{itemize}
By (\ref{eq:mass of good cylinders3}), the fact that $\int f_{i}\:d\mu_{i}\le h_{i}+\delta$,
Egorov's Theorem and the ergodicity of $\mu_{i}$, we may assume that
$\mu_{i}(\Omega_{0,i})>1-O(\delta)$. \begin{lemma} There exists
a global constant $c_{2}>1$ such that for $i=1,2$, $\omega\in\Omega_{0,i}$,
and $n\ge N$, 
\begin{equation}
-\frac{1}{n_{i}m}\log\sigma_{i}^{j}\nu_{i}[\omega|_{n_{i}m}]\le h_{i}+c_{2}\delta/\epsilon\text{ for some }0\le j<m\:.\label{eq:some j3}
\end{equation}
\end{lemma} \begin{proof} The proof uses exactly the same method
as the proof of Lemma \ref{nergodic}. \end{proof} Let us resume
with the proof of the theorem. For $i=1,2$ and a Borel set $E\subset\Omega_{i}$
write $\mu_{0,i}(E)=\frac{\mu_{i}(E\cap\Omega_{0,i})}{\mu_{i}(\Omega_{0,i})}$,
and write $\theta_{0}$ for $\Pi_{1}\mu_{0,1}*\Pi_{2}\mu_{0,2}$.
By Lemma \ref{lemma:upper semi cont} the function which takes a probability measure $\zeta$ on $\mathbb{R}$
to $\dim_{H}\zeta$ is upper semi-continuous with respect to the total
variation distance. From $\mu_{i}(\Omega_{0,i})>1-O(\delta)$
it follows that the total variation distance between $\theta$ and
$\theta_{0}$ is $O(\delta)$. Thus we may assume that, 
\[
\dim_{H}\theta_{0}\le\dim_{H}\theta+\epsilon/2<\beta-\epsilon/2\:.
\]
Let $n\ge N$ and $x\in\mathbb{R}$ be with, 
\[
\frac{\log\theta_{0}(B(x,2^{-nm}))}{-nm}<\beta-\epsilon/2\:.
\]
Let $g:\Omega_{1}\times\Omega_{2}\rightarrow\mathbb{R}$ be with $g(\omega_{1},\omega_{2})=\Pi_{1}\omega_{1}+\Pi_{2}\omega_{2}$,
then $\theta_{0}=g(\mu_{0,1}\times\mu_{0,2})$. Let $\mathcal{U}$
be the set of all pairs of words $(w_{1},w_{2})\in\Lambda^{n_{1}m}\times\Lambda^{n_{2}m}$
such that 
\[
([w_{1}]\times[w_{2}])\cap g^{-1}(B(x,2^{-nm}))\ne\emptyset,
\]
and $\mu_{0,i}[w_{i}]>0$ for $i=1,2$.

Since $\mu_{i}(\Omega_{0,i})>1-O(\delta)>1/2$, 
\begin{multline}
2^{-nm(\beta-\epsilon/2)}<\theta_{0}(B(x,2^{-nm}))=g(\mu_{0,1}\times\mu_{0,2})(B(x,2^{-nm}))\\
\le\sum_{(w_{1},w_{2})\in\mathcal{U}}\mu_{0,1}[w_{1}]\mu_{0,2}[w_{2}]\le4\sum_{(w_{1},w_{2})\in\mathcal{U}}\mu_{1}[w_{1}]\mu_{2}[w_{2}]\:.\label{eq:lb on sum of mass of cyl3}
\end{multline}
For each $(w_{1},w_{2})\in\mathcal{U}$ we have $\mu_{0,i}[w_{i}]>0$
for $i=1,2$, hence $\Omega_{0,i}\cap[w_{i}]\ne\emptyset$, and so
$\mu_{i}[w_{i}]<2^{-n_{i}m(h_{i}-\delta)}$. From this and (\ref{eq:lb on sum of mass of cyl3})
we get, 
\[
2^{-nm(\beta-\epsilon/2)}<\exp(2-n_{1}mh_{1}-n_{2}mh_{2}+\delta m(n_{1}+n_{2}))\cdot|\mathcal{U}|\:.
\]

For $0\le j_{1},j_{2}<m$ write, 
\[
\mathcal{U}_{j_{1},j_{2}}=\{(w_{1},w_{2})\in\mathcal{U}\::\:\sigma_{i}^{j_{i}}\nu_{i}[w_{i}]\ge\exp(-n_{i}m(h_{i}+c_{2}\delta/\epsilon))\text{ for }i=1,2\}\:.
\]
From (\ref{eq:some j3}) and $n\ge N$, and since $\Omega_{0,i}\cap[w_{i}]\ne\emptyset$
for $i=1,2$ and $(w_{1},w_{2})\in\mathcal{U}$, it follows that $\mathcal{U}=\cup_{j_{1},j_{2}=0}^{m-1}\mathcal{U}_{j_{1},j_{2}}$.
Hence there exist $0\le j_{1},j_{2}<m$ with 
\begin{equation}
|\mathcal{U}_{j_{1},j_{2}}|\ge|\mathcal{U}|/m^{2}>\frac{1}{4m^{2}}\exp(n_{1}mh_{1}+n_{2}mh_{2}-nm(\beta-\epsilon/2)-\delta m(n_{1}+n_{2}))\:.\label{eq:lb on card U_j3}
\end{equation}

For $i=1,2$ let $K_{i}$ be the attractor of $\Phi_{i}$. Without
loss of generality we may assume that $\mathrm{diam}(K_{i})\le1$.
Given $(w_{1},w_{2})\in\mathcal{U}_{j_{1},j_{2}}$ we have 
\[
g([w_{1}]\times[w_{2}])\cap B(x,2^{-nm})\ne\emptyset\:.
\]
Also, since $n_{i}=\left\lceil \frac{n}{-\log r_{i}}\right\rceil $,
\begin{multline*}
\mathrm{diam}(g([w_{1}]\times[w_{2}]))=\mathrm{diam}(\Pi_{1}[w_{1}])+\mathrm{diam}(\Pi_{2}[w_{2}])\\
=\mathrm{diam}(\varphi_{w_{1},1}(K_{1}))+\mathrm{diam}(\varphi_{w_{2},2}(K_{2}))\le r_{1}^{n_{1}m}+r_{2}^{n_{2}m}\le2^{1-nm},
\end{multline*}
which implies that 
\[
[w_{1}]\times[w_{2}]\subset g^{-1}(B(x,2^{2-nm}))\:.
\]
Hence, by the definition of $\mathcal{U}_{j_{1},j_{2}}$, 
\begin{eqnarray*}
g(\sigma_{1}^{j_{1}}\nu_{1}\times\sigma_{2}^{j_{2}}\nu_{2})(B(x,2^{2-nm})) & \ge & \sigma_{1}^{j_{1}}\nu_{1}\times\sigma_{2}^{j_{2}}\nu_{2}(\cup_{(w_{1},w_{2})\in\mathcal{U}_{j_{1},j_{2}}}[w_{1}]\times[w_{2}])\\
 & \ge & |\mathcal{U}_{j_{1},j_{2}}|\cdot\exp(-n_{1}mh_{1}-n_{2}mh_{2}-(n_{1}+n_{2})mc_{2}\delta/\epsilon)\:.
\end{eqnarray*}
From this and (\ref{eq:lb on card U_j3}), 
\[
g(\sigma_{1}^{j_{1}}\nu_{1}\times\sigma_{2}^{j_{2}}\nu_{2})(B(x,2^{2-nm}))\ge\frac{1}{4m^{2}}\exp\left(-nm\left(\beta-\epsilon/2+O_{r_{1},r_{2}}(\delta/\epsilon)\right)\right)\:.
\]

On the other hand, by (\ref{eq:all j mass of balls small3}) and by
assuming that $n$ is large enough, 
\[
g(\sigma_{1}^{j_{1}}\nu_{1}\times\sigma_{2}^{j_{2}}\nu_{2})(B(x,2^{2-nm}))\le\exp((2-nm)(1-\delta)\alpha)\:.
\]
Hence 
\[
\frac{1}{4m^{2}}\exp\left(-nm\left(\beta-\epsilon/2+O_{r_{1},r_{2}}(\delta/\epsilon)\right)\right)\le\exp((2-nm)(1-\delta)\alpha),
\]
and so by taking logarithm on both sides, dividing by $-nm$, and
letting $n$ tend to $\infty$, we get 
\[
\beta-\epsilon/2+O_{r_{1},r_{2}}(\delta/\epsilon)\ge(1-\delta)\alpha\:.
\]
Since this holds for every $0<\alpha<\beta-c_{1}\delta$, 
\begin{equation}
\beta-\epsilon/2+O_{r_{1},r_{2}}(\delta/\epsilon)\ge(1-\delta)(\beta-c_{1}\delta)\:.\label{eq:contradiction3}
\end{equation}
Now recall that $\delta$ is arbitrarily small with respect to $\epsilon$,
and so (\ref{eq:contradiction3}) gives a contradiction. Thus we must
have $\dim_{H}\theta\ge\beta$, which completes the proof of the theorem.

\section{Orthogonal projections of ergodic measures}

In this section we show how to use the ideas above in order to prove
a result on the orthogonal projections of ergodic measures. As in
previous sections, the main ingredient in the proof is a result from
\cite{Sh}. 

Let $U$ be a $2\times2$ orthogonal matrix with $U^{n}\ne Id$ for
all $n\ge1$ and let $0<r<1$. Let $\Phi=\{\varphi_{\lambda}(x)=rUx+a_{\lambda}\}_{\lambda\in\Lambda}$
be a self-similar IFS on $\mathbb{R}^{2}$. Suppose that $\Phi$ satisfies
the open set condition. Let $S^{1}$ be the unit circle of $\mathbb{R}^{2}$.
For $z\in S^{1}$ and $y\in\mathbb{R}^{2}$ write $P_{z}y=\left\langle z,y\right\rangle $.
Write $\Omega=\Lambda^{\mathbb{N}}$, let $\sigma:\Omega\rightarrow\Omega$
be the left shift, and let $\Pi:\Omega\rightarrow K$ be the coding
map for $\Phi$. \begin{thm}\label{OP} Let $\mu$ be a $\sigma$-invariant
and ergodic measure on $\Omega$. Write $h$ for the entropy of $\mu$.
Then for every $z\in S^{1}$ the measure $P_{z}\Pi\mu$ is exact dimensional
and 
\[
\dim P_{z}\Pi\mu=\min\{1,\frac{h}{-\log r}\}\:.
\]
\end{thm} 
In Theorem 1.6 in \cite{HS} the above result is shown
for self-similar measures and it is shown for Gibbs measures in \cite{BJ}. The methods used in \cite{HS} and \cite{BJ} do not seem to adapt to general ergodic measure. However the results in \cite{BJ} do work for Gibbs measures on self-conformal sets as well as on self-similar sets. We do not know how to extend our results to the setting of self-conformal sets.

\subsection*{Sketch of the proof of Theorem \ref{OP}}

The proof is almost identical to the ones given for Theorems \ref{thm:main}
and \ref{thm:conv}, thus we only provide a short sketch.

Let $\beta=\min\{1,\frac{h}{-\log r}\}$, then it suffices to show
that $\dim_{H}P_{z}\Pi\mu\ge\beta$ for all $z\in S^{1}$. Assume
by contradiction that there exists $z\in S^{1}$ with $\dim_{H}P_{z}\Pi\mu<\beta$.
Let $0<\epsilon<\beta-\dim P_{z}\Pi\mu$ be small in a manner depending
on $\Phi$ and $\mu$, let $\delta>0$ be small with respect to $\epsilon$,
and let $m\ge1$ be large with respect to $\delta$.

Next we construct a Bernoulli measure $\nu$ which corresponds to
$\mu$ as in the proof of Theorem \ref{thm:conv}. Namely, write
\[
\mathcal{W}=\{w\in\Lambda^{m}\::\:2^{-m(h+\delta)}\le\mu[w]\le2^{-m(h-\delta)}\},
\]
for $w\in\Lambda^{m}$ set 
\[
p_{w}=\begin{cases}
\mu[w]\cdot c & \text{if }w\in\mathcal{W}\\
2^{-m\epsilon^{-1}}\cdot c & \text{otherwise}
\end{cases}
\]
(where $1/2\le c\le2$ is a normalizing constant), and let $\nu$
be the measure on $\Omega$ with, 
\[
\nu[w_{1}...w_{l}]=p_{w_{1}}\cdot...\cdot p_{w_{l}}\text{ for each }w_{1},...,w_{l}\in\Lambda^{m}\:.
\]

Write $q$ for $\delta^{-1}$, and recall that given a Borel probability
measure $\zeta$ on $\mathbb{R}$ its $L^{q}$ dimension is denoted
by $D(\zeta,q)$.\begin{lemma} There exists a constant $c_{1}\ge1$,
which depends only on $r$, such that 
\begin{equation}
D(P_{v}\Pi\nu,q)>\beta-c_{1}\delta\text{ for all }v\in S^{1}\:.\label{eq:L^q dim of proj of nu2}
\end{equation}
\end{lemma}
\begin{proof}
We have
\[
\Vert p\Vert_{q}^{q}\le\Vert p\Vert_{\infty}^{q-1}\sum_{w\in\Lambda^{m}}p_{w}\le\exp(-m(h-\delta)(q-1))\:.
\]
From this and \cite[Theorem 8.2]{Sh} it follows that for all $v\in S^{1}$,
\[
D(P_{v}\Pi\nu,q)=\min\{1,\frac{\log\Vert p\Vert_{q}^{q}}{(q-1)\log r^{m}}\}\ge\min\{1,\frac{h-\delta}{-\log r}\},
\]
which completes the proof of the lemma.
\end{proof}
Fix some $0<\alpha<\beta-c_{1}\delta$. \begin{lemma} There exists
$\eta_{0}>0$, which depends on all previous parameters, such that
for every $0\le j<m$, 
\begin{equation}
P_{z}\Pi\sigma^{j}\nu(B(x,\eta))\le\eta^{(1-\delta)\alpha}\text{ for all }0<\eta\le\eta_{0}\text{ and }x\in\mathbb{R}\:.\label{eq:all j mass of balls small2}
\end{equation}
\end{lemma} \begin{proof} Write, 
\[
T=\{U^{j}z\::\:0\le j<m\}\:.
\]
By \cite[Lemma 1.7]{Sh}, (\ref{eq:L^q dim of proj of nu2}), and
$q=\delta^{-1}$, it follows that there exists $\eta_{1}>0$ with,
\[
P_{v}\Pi\nu(B(x,\eta))\le\eta^{(1-\delta)(\beta-c_{1}\delta)}\text{ for all }v\in T,\:0<\eta\le\eta_{1}\text{ and }x\in\mathbb{R}\:.
\]
\par Let $\eta_{0}>0$ be small with respect to $\eta_{1}$ and all
previous parameters. Let $0\le j<m$, $0<\eta\le\eta_{0}$, $x\in\mathbb{R}$,
and $u\in\Lambda^{j}$. Write $b=\left\langle z,U^{-j}\varphi_{u}(0)\right\rangle $,
then 
\begin{eqnarray*}
P_{z}\Pi\sigma^{j}(\nu|_{[u]})(B(x,\eta)) & = & \nu\{\omega\in[u]\::\:P_{z}\Pi\sigma^{j}\omega\in B(x,\eta)\}\\
 & = & \nu\{\omega\in[u]\::\:P_{z}\varphi_{u}^{-1}\Pi\omega\in B(x,\eta)\}\\
 & = & \nu\{\omega\in[u]\::\:P_{U^{j}z}\Pi\omega\in B(x+b,r^{j}\eta)\}\\
 & \le & P_{U^{j}z}\Pi\nu(B(x+b,r^{j}\eta))\le\eta^{(1-\delta)(\beta-c_{1}\delta)}\:.
\end{eqnarray*}
Hence, 
\[
P_{z}\Pi\sigma^{j}\nu(B(x,\eta))=\sum_{u\in\Lambda^{j}}P_{z}\Pi\sigma^{j}(\nu|_{[u]})(B(x,\eta))\le|\Lambda|^{m}\eta^{(1-\delta)(\beta-c_{1}\delta)}<\eta^{(1-\delta)\alpha},
\]
which completes the proof of the lemma. \end{proof} After this point
the argument proceeds exactly as in the proofs of Theorems \ref{thm:main}
and \ref{thm:conv}, and we can complete the proof of Theorem \ref{OP}.

\section{Applications and remarks}

For a self-similar set the similarity dimension $s$ is defined to
be the unique solution of $\sum_{\lambda\in\Lambda}r_{\lambda}^{s}=1$.
In the case where the similarity dimension is less than or equal to
$1$ there is a more straightforward proof for Theorem \ref{thm:main},
where $\nu$ can simply be taken to be the self-similar measure with
weight $r_{\lambda}^{s}$ for each $\lambda\in\Lambda$. In fact with this assumption Theorem \ref{thm:main} can be extended to show no dimension drop for non-invariant measures and sets, where the dimension on the symbolic space is defined to be compatible with the self-similar set.

We can also give a general bound on the dimension of an ergodic measure $\mu$, projected to a self-similar set, in terms of $L^q$ dimensions. As above we let $\nu$ be the self-similar measure with
weight $r_{\lambda}^{s}$ for each $\lambda\in\Lambda$, where $s$ is the similarity dimension. If for $q>1$ we let $D_{\Pi\nu}(q)$ denote the $L^q$ dimension of $\Pi\nu$ and $\alpha_{\min}=\lim_{q\to\infty}D_{\Pi\nu}(q)$, then for any $0<\alpha<\alpha_{\min}$ there exists $C>0$ such that
$$\Pi\nu(B(x,r))\leq Cr^{\alpha}\textit{ for all }x\in\mathbb{R}\textit{ and }r>0\:.$$
Now by using some of the ideas appearing in the proofs above, it can be shown that
$$\dim \Pi\mu\geq \frac{h(\mu)}{-\chi_{\mu}}-(s-\alpha_{\min}).$$
This can be applied in situations where exponential separation is not satisfied but the $L^q$ spectrum is known, for examples of this see \cite{Fe}. In the case where $s>1$ it may be possible to adapt the methods given earlier to produce better methods, but this will be very dependent on the specific system.

If we have a diagonal self-affine system in the plane, satisfying suitable separation conditions, then we can combine our
Theorem \ref{thm:main} with Theorem 2.11 in \cite{FH} to show that
the dimension of any ergodic measure will be the Lyapunov dimension
(the Lyapunov dimension is the natural generalisation of the entropy
divided by Lyapunov exponent formula for ergodic measures projected
on self-affine systems). 

To give the full details of this let $\Lambda$ be a finite non-empty set and for each $\lambda\in\Lambda$ let  $\varphi_{\lambda}:\R^2\to\R^2$ be given by $\varphi_{\lambda}(x,y)=(a_{\lambda} x+s_{\lambda},b_{\lambda} y+t_{\lambda})$ where $0<|a_{\lambda}|,|b_{\lambda}|<1$ and $s_{\lambda},t_{\lambda}\in\R$. In this setting there exists a unique non-empty compact set $K$ such that $K=\cup_{\lambda\in\Lambda} \varphi_{\lambda}(K)$. Let $\Omega=\Lambda^{\mathbb{N}}$ and denote by $\Pi:\Omega\to K$ the natural projection to the self-affine set. We assume that this map is finite to one. In particular this is satisfied when the strong separation condition holds, in which case $\Pi$ is injective.

Let $\pi_x:\Omega\to\R$ denote the projection to the self-similar set given by $\{a_{\lambda}x+s_{\lambda}\}_{\lambda\in\Lambda}$ and 
$\pi_y:\Omega\to\R$ denote the projection to the self-similar set given by $\{b_{\lambda}y+t_{\lambda}\}_{\lambda\in\Lambda}$. 
For a fixed ergodic measure $\mu$ on $\Omega$ we let $\chi_x(\mu)$, $\chi_y(\mu)>0$ denote the Lyapunov exponents with respect to the respective self-similar systems. We also let $\chi_1(\mu)=\min\{\chi_x(\mu),\chi_y(\mu)\}$ and   
$\chi_2(\mu)=\max\{\chi_x(\mu),\chi_y(\mu)\}$, and let $\pi_1$ be the projection which corresponds to the smaller Lyapunov exponent $\chi_1(\mu)$.

In this setting Theorem 2.11 in \cite{FH} gives that 
$$\dim_H \Pi\mu=\frac{h_{\pi_1}(\mu)}{\chi_1(\mu)}+\frac{h(\mu)-h_{\pi_1}(\mu)}{\chi_2(\mu)}.$$
Here $h(\mu)$ is the usual entropy and $h_{\pi_1}(\mu)$ is the projected entropy which satisfies $\dim(\pi_1\mu)\chi_1(\mu)=h_{\pi_1}(\mu)$ (see Theorem 2.8 in \cite{FH}). Now suppose that the direction corresponding to the smaller Lyapunov exponent satisfies exponential separation. Then Theorem \ref{thm:main} gives that if $\chi_1(\mu)\geq h(\mu)$ then $h(\mu)=h_{\pi_1}(\mu)$ and so 
$$\dim_H \Pi\mu=\frac{h(\mu)}{\chi_1(\mu)}.$$
On the other hand if $\chi_1(\mu)< h(\mu)$ then Theorem \ref{thm:main} gives that $\chi_1(\mu)=h_{\pi_1}(\mu)$ and so
$$\dim_H \Pi\mu=1+\frac{h(\mu)-\chi_1(\mu)}{\chi_2(\mu)}.$$
This means that whenever the self-similar set corresponding to the smaller Lyapunov exponent satisfies exponential separation,
$$\dim_H \pi\mu=\min\left\{\frac{h(\mu)}{\chi_1(\mu)},1+\frac{h(\mu)-\chi_1(\mu)}{\chi_2(\mu)}\right\}$$
which is what we required.
  
Note that the requirement of $\Pi:\Omega\to X$ being finite to $1$ is used to show that the projected entropy in Theorem 2.11 in \cite{FH} is the same as the usual entropy. It should be possible to weaken this assumption considerably. For example to a suitable exponential separation condition for the diagonal self-affine set.
\subsection*{Acknowledgment}

A.R. acknowledges support from the Herchel Smith Fund at the University
of Cambridge. We also thank Pablo Shmerkin for suggesting the application
to convolutions of ergodic measures. We would also like to thank the anonymous referee for making several helpful suggestions to improve the manuscript.

\end{document}